\documentclass[12pt]{article}
\usepackage{amsfonts}
\usepackage{amsmath}
\usepackage{amsthm}
\usepackage{graphicx}
\newtheorem{thm}{Theorem}[section]

\newtheorem{lem}[thm]{Lemma}
\newtheorem{prop}[thm]{Proposition}

\theoremstyle{definition}
\newtheorem{defn}[thm]{Definition}

\theoremstyle{remark}
\newtheorem{rem}{Remark}

\numberwithin{equation}{section}

\font\nt=cmr7

\def\note#1
{\marginpar
{\nt $\leftarrow$
\par
\hfuzz=20pt \hbadness=9000 \hyphenpenalty=-100 \exhyphenpenalty=-100
\pretolerance=-1 \tolerance=9999 \doublehyphendemerits=-100000
\finalhyphendemerits=-100000 \baselineskip=6pt
#1}\hfuzz=1pt}

\def\be{\begin{equation}}
\def\ee{\end{equation}}

\newcommand{\ra}{\rightarrow}

\newcommand{\dist}{\operatorname{dist}}

\renewcommand{\AA}{{\cal A}}

\newcommand{\MM}{{\cal M}}

\newcommand{\PP}{{\cal P}}

\renewcommand{\SS}{{\cal S}}

\newcommand{\C}{{\mathbb C}}

\newcommand{\N}{{\mathbb N}}

\newcommand{\R}{{\mathbb R}}

\newcommand{\Z}{{\mathbb Z}}

\def\B0{{\mathbf{0}}}

\renewcommand{\Im}{\operatorname{Im}}

\newcommand{\ov}{\overline}

\renewcommand{\ra}{\rightarrow}

\newcommand{\ctzero}{{\tilde{c_0}}}
\newcommand{\czero}{{c_0}}
\newcommand{\Suno}{{\mathbb{S}^1}}

\catcode`\@=12

\def\Empty{}
\newcommand\oplabel[1]{
  \def\OpArg{#1} \ifx \OpArg\Empty {} \else
  	\label{#1}
  \fi}
		
\newcommand{\comm}[1]{}
\newcommand{\comment}[1]{}

\newcommand{\Wij}{\widehat{W}_{i,j}}
\begin{document} 

 \title{Triviality of fibers for Misiurewicz parameters in the exponential family}
  \author{Anna Miriam Benini}
\maketitle

\begin{abstract}  
We consider the family of holomorphic maps $e^z+c$ and show that fibers of postsingularly finite parameters are trivial. This can be considered as the first and simplest class of non-escaping parameters for which we can obtain results about triviality of fibers in the exponential family.
\end{abstract} 

\section{Introduction}
Among transcendental  functions, the complex exponential family $e^z+c$ is one of the most widely studied examples, because there is only one singular value at $c$ and because in some respects it can be seen as combinatorial limit of unicritical polynomials of degree $d$ (See e.g. \cite{D}).

As usual in one-dimensional complex dynamics, if $f:\C\ra\C$ is a holomorphic function,  investigation of the dynamics of $f$ starts with the definition of the Fatou set
\[F(f):=\left\{z\in\C, \{f^n\}\text{ is a normal family in a neighborhood of }z\right\}\]
and of the Julia set $J(f)$ as the complement of the Fatou set.
 

The foundational work in the study of the exponential family has been set by Baker and Rippon \cite{BR}  and by Eremenko and Lyubich \cite{EL}; among other properties, they prove (for exponential maps and for more general entire functions respectively) that there are no wandering Fatou components for $e^z+c$. As a corollary, if there are no attracting orbits, parabolic orbits or Siegel disks, the Julia set is equal to $\C$. 

We will make use of the combinatorial structure of dynamical and parameter plane  worked out by Schleicher, Zimmer and Rempe; however, many other aspects have been investigated, among others, by Karpinska, Urbanski, Devaney and coauthors. For a review of exponential dynamics and a more complete set of references see for example \cite{R0}.

As a parameter analogy to the Fatou set, given a one-dimensional family of holomorphic maps $\{g_c\}_{c\in\C}$  we can define the \emph{set of structurally stable parameters}

\begin{small}
\[ S= \{c_0\in\C, g_{c_0} \text{ is topologically conjugate to }g_c \text{ $\forall c$ in a neighborhood of }c_0\}.\]
\end{small}
By a theorem of Eremenko and Lyubich (\cite{EL}, Theorem 10) $S$ is open and dense in $\C$ for the family $e^z+c$.

The complement of $S$ is called the \emph{bifurcation locus}.
 
One of the central problems for one-parameter families in one dimensional complex dynamics is to establish whether  each   function $g_c$ which is \emph{structurally stable} is also \emph{hyperbolic}, i.e. has an attracting cycle (see Section \ref{Fibers and Rigidity} for a wider discussion).
Components of $S$ which do not contain hyperbolic parameters are called \emph{non-hyperbolic}.

The families of functions for which we have the most results about this conjecture are the families of unicritical polynomials of degree $D$ usually parametrized as $P_c^D(z)=z^D+c$. 
 Most proofs of this kind of results involve a construction called Yoccoz puzzle and estimates on the modulus of the annuli between puzzle pieces; in the  exponential family most dynamically arising objects including puzzle pieces are unbounded, so that the corresponding annuli are degenerate at infinity breaking down the general strategy of the proofs.
 
A different way of approaching this problem is studying combinatorial properties (see Section \ref{Combinatorics and Ray Portraits}) of maps which are topologically conjugate. In particular the structure of periodic dynamic rays landing together (see Section \ref{Dynamic and parameter rays} for definition of rays, and Section \ref{Combinatorics and Ray Portraits} for a discussion about rays landing together) is a topological invariant. Dynamic rays are labeled by sequences in $\Z^\N$ called \emph{addresses} (see again Section \ref{Dynamic and parameter rays}), which are the analog of angles for polynomials dynamic rays.
 
In this setting, it is possible to state that for some specific parameters the structure of periodic dynamic rays landing together (the \emph{combinatorics}) characterizes them uniquely, so that they cannot be topologically conjugate to any other parameter nor be on the boundary of a component of $S$ which is not hyperbolic; we usually refer to this kind of results as \emph{rigidity} results. This brings to the definition of \emph{parameter fibers} as set of parameters which all have the same combinatorics (see Section \ref{Fibers and Rigidity}).  By the theory of parabolic bifurcations from \cite{RS2}, for two parameters having the same combinatorics is equivalent to not being separated by any pair of periodic  parameter rays landing together.
 
 This paper investigates the properties of parameter fibers of parameters which are \textsl{postsingularly finite},  i.e those parameters $c$ for which the set of forward images of the singular value $c$ is finite under iteration of $f_c(z)= e^z+c$. These parameters are commonly called \emph{Misiurewicz parameters} in holomorphic dynamics.
 
 Our main result is stated as follows:
\begin{thm}\label{Triviality of fibers}
 Parameter fibers of Misiurewicz parameters  are trivial, i.e. given any Misiurewicz parameter $c_0$, for any other parameter $c$ which does not belong to one of the finitely many parameter rays landing at $c_0$ there is a pair of parameter rays with periodic addresses, landing together at a parabolic parameter, which separate $c$ from $c_0$.
 \end{thm}
  We will devote the first section to a collection of relevant results about existence and landing properties of dynamic and parameter rays for the exponential family. The second section will introduce Misiurewicz parameters and their combinatorial properties, followed by a section on orbit portraits where we will prove some explicit theorems about the correspondence of orbit portraits between exponentials and polynomials. After that we will give a short introduction to fibers and rigidity, and in the last section we will present the statement and the proof of Theorem \ref{Triviality of fibers}.
  
For the exponential family, we will refer as $\Pi_P$ to the parameter plane and as $\Pi_c$ to the dynamical plane for the parameter $c$.

We will indicate by $P_D$ the family of unicritical polynomials of degree $D$, call $\Pi_P^D$ their parameter plane and $\Pi_c^D$ the dynamical plane with the dynamics given by $P^D_c(z)=z^D+c$.

Many thanks are due to Mikhail Lyubich and Dierk Schleicher for suggesting this problem, and to Lasse Rempe, Dierk Schleicher and especially Mikhail Lyubich for helpful discussions on the subject. 
 The pictures have been drawn with the program \emph{It} written by Christian Mannes.

\section{Dynamic rays and parameter rays}\label{Dynamic and parameter rays}
This section has the purpose of recollecting some of the relevant results about rays and their landing properties.
We will briefly introduce dynamical rays (originally called external rays) and parameter rays for unicritical polynomials first and for exponential maps then; subsequently we will present some of the properties that we need. 
We will use the same notation for exponentials and polynomials, in order to make it easier to state some theorems in parallel.

Dynamical and parameter rays for polynomials are a very classical topic in complex dynamics; excellent references are \cite{M}, Chapter 18, for dynamic rays;  \cite{CG}, Chapter 8 and \cite{PR} for parameter rays. 

Let  $P^D_c(z)=z^D+c$ be a unicritical polynomial. This can be seen as a map from $\hat{\C}$ to $\hat{\C}$ with  infinity being a superattracting fixed point. By B\"ottcher's theorem, there is a holomorphic function $B_c$, tangent to the identity at infinity, conjugating the dynamics of $P^D_c(z)$ in a neighborhood of $\infty$ containing the critical  value $c$ to the dynamics of $P^D_0$ (also in a neighborhood of infinity). 

For $s\in \Suno$, call $R_s:=\{z\in \C, z=r e^{i s}, r>1\}$ the straight ray of angle $s$, and  let $s_c\in\Suno$ be the angle such that $B_c(c)\in R_{s_c}$.  If $s$ is written in $D$-iadic  expansion, and $\sigma$ denotes  the shift map, the dynamics of $P^D_0$ carries $R_s$ to $R_{\sigma s}$.

Using the functional equation, $B_c$ can be extended to a conjugacy $\widetilde{B}_c$ defined on $\C$ minus the preimages under iterates of $P^D_c$ of the set $B^{-1}_c( R_{s_c})$.
Given a straight ray $R_s$ such that $\sigma^k(s)\neq s_c$ for any $k\in\N$, the set \[g^c_s:=\widetilde{B}^{-1}_c (R_s)\] is well defined curve and is called the \emph{dynamic ray}  of angle $s$ (for the polynomial $P^D_c$). It can be parameterized by a parameter $t$ called \emph{potential} so that 
\[P^D_c(g^c_s)(t)=g^c_{\sigma s}(D t).\]

For $s\in\Suno$, the set $G_s:=\{c\in\C, c\in g^c_s \}$ can be shown to be a simple curve and is called the \emph{parameter ray} of angle $s$.

Now let us define dynamic and parameter rays for the exponential family.
We will use throughout the paper the concept of itinerary with respect to a partition:
\begin{defn}
Let $f:\C\ra\C$ be a function, and $\MM= \{ M_{a_i} \}_{a_i\in \AA}$ be a countable collection of pairwise disconnected domains of $\C$ such that each domain is labeled uniquely by a letter in a countable alphabet $\AA$. 
If $f^j(z)\in\underset{a_i\in\AA}\cup M_{a_i}$ for all $j\in\N\cup\{0\}$, then  we say that the \emph{itinerary} of $z$ (with respect to $\MM$) is  the sequence $a=a_1...a_n...$ defined by  $f^j(z)\in M_{a_j}$.
%
\end{defn}
Dynamic rays for the exponential family (\cite{DK},\cite{SZ3}) have been introduced in analogy with the polynomial case in order to construct symbolic dynamics on the set of \emph{escaping points} 
\[I(f_c):=\{z\in\C, |f_c^n(z)| \rightarrow \infty \text{ as } n\ra\infty\}\subset \Pi_c.\]

Let
\[ S_j:=\{z\in \C, (2j -1)\pi<\Im z<(2j +1)\pi\},\]
 and consider itineraries of points with respect to this partition, i.e. \[\text{itin}(z)=s_1 s_2\dots\text{ if and only if }f^j(z)\in S_{s_j}\] for points whose iterates never belong to the boundaries of the partition. 
 
%
 
For any such point $z$  we will refer to its itinerary with respect to the strips $S_j$ as the \emph{ address}\footnote{This is often referred to as \emph{external address} to distinguish it from other kinds of addresses. However, this is  the only kind of address we deal with in this paper, so the term \emph{address} will not create ambiguity.} of $z$.


Using the described construction itineraries of points cannot have entries  growing faster than iterates of the  exponential function. This  leads to the following notion.
\begin{defn}
A sequence $s=s_1 s_2\dots$ is called \emph{exponentially bounded}  if there exists $ x\in \R$ such that $ | 2\pi s_j|<F^j(x)$
for all $j\in \N$,  where $F:\R \mapsto \R$ is the growth function  $F:t\mapsto e^t-1$.
\end{defn} 
This growth condition turns out to be not only necessary but also sufficient \cite{SZ2}, so that  all sequences $s$ contained in the set 
 \[\SS:=\{s\in \Z^{\N}, s \text{ is exponentially bounded}  \}\]
 are realized as itineraries of some point $z$.
 
 We will say that an address is \emph{periodic} if it is a periodic sequence,\emph{ preperiodic} if it is a strictly preperiodic sequence and \emph{(pre)periodic} if it is either periodic or strictly preperiodic. 
 
 The set $\SS$ has a natural order induced by the usual order relation on the space of sequences over an ordered set. 
 
 If $s=s_1 s_2\dots,$ we will say that $|s|=\underset{i}\sup |s_i|$.

 Given a an external address $s=s_1 s_2...$ we will define its \emph{minimal potential} 
\[t_s:=\inf\left\{t>0,\lim\underset{k\geq1}\sup \frac{|s_k|}{F^{k}(t)}=0  \right\}.\]

Definition, existence and properties of \emph{dynamic rays} for the exponential family are summarized in  the following theorem (\cite{SZ2}, Proposition 3.2 and Theorem 4.2):
\begin{thm}{\bf Existence of dynamic rays.}\label{Existence of dynamic rays}
Let $c$ be a parameter such that $|f^n_c(c)|$ does not tend to infinity as $n\ra \infty$; then 
for any $s\in \SS$ there exists a unique injective curve $g^c_s: (t_s,\infty) \rightarrow \C$ consisting of escaping points such that 
\begin{itemize}
\item
$g^c_s (t)$ has  address $s$ for sufficiently large $t$;
\item
$f_c (g^c_s(t))=g^c_{\sigma s}(F(t)) $;
\item
We have the asymptotics 
$g^c_s(t)=2\pi i s_1 + t+ o(e^{-t})$ as $t\ra \infty$. 
\end{itemize}
\end{thm}
Given  $s\in\SS$, we will call the unique curve $g_s^c$ given by Theorem \ref{Existence of dynamic rays} the \emph{dynamic ray of address $s$}.

Like for polynomials, a  dynamic ray $g_s$ is called \textsl{periodic} if $f^m_c(g_s)\subset g_s$ for some $m\in \N$, and is called \textsl{preperiodic} if $f^{m+k}_c(g_s)\subset f_c^k(g_s)$ for $m,k>1$. Also, a dynamic ray is (pre)periodic if and only if its address is (pre)periodic.

The question whether periodic dynamic rays land for the exponential family remained open for some time, and was finally solved by Rempe using the previously known fact that periodic rays land for hyperbolic parameters and an argument about persistence of landing inside wakes. This led to the following Theorem (\cite{R1}, Theorem 1):

\begin{thm}{\bf Landing theorem for periodic dynamic rays.}\label{Landing theorem for periodic dynamic rays}
Let $c$ be such that $|f^n_c(c)|$ does not tend to infinity as $n\ra \infty$.
Then every periodic dynamic ray $g^c_s$ lands at a repelling or parabolic periodic point; also every preperiodic dynamic ray lands at a preperiodic point.
\end{thm}


The construction of parameter rays is also done keeping in mind the fundamental property that  parameter rays have for polynomials: a point $c$ belongs to some parameter ray $G_s$ in $\Pi_P$ if and only if $c$ belongs to the dynamic ray $g^c_s$ in $\Pi_c$. It is carried out by 
F\"orster and Schleicher and is summarized in the following theorem about existence of parameter rays (\cite{FS},Theorem 3.7): 

\begin{thm}{\bf Existence of parameter rays.}\label{Existence of parameter rays}
Let $s\in\SS$. Then there is a unique injective curve $G_s: (t_s,\infty)\ra\C$, such that, for
all $t > t_s$ , $c=G_s(t)$ if and only if $c=g^c_s(t)$. 

The map $G_s: (t_s,\infty)\rightarrow \C$ is continuous, and $|G_s(t)-(t+2\pi i s_1)|\rightarrow 0$ as $t\rightarrow \infty.$ 
\end{thm}
Given an address $s\in\SS$, we will call the unique curve $G_s$ given by Theorem \ref{Existence of parameter rays} the \emph{parameter ray of address $s$}.
%

Not much is known about the landing properties of parameter rays. However it is known (\cite{S3}) that parameter rays of periodic and preperiodic addresses land at parabolic and Misiurewicz parameters respectively.

\subsubsection*{Combinatorial spaces and cyclic order}\label{Combinatorial spaces and cyclic order}

Theorems \ref{Existence of dynamic rays} and \ref{Existence of parameter rays} establish a correspondence between the set of dynamic/parameter rays and the set of exponentially bounded addresses $\SS$ for the exponential family. Moreover, the equation  $f_c (g^c_s(t))=g^c_{\sigma s}(F(t))$ in Theorem \ref{Existence of dynamic rays} tells that the dynamics of an exponential function $f_c$ on its set of escaping points $I$ is conjugate to the dynamics of the left-sided shift map $\sigma$ on $\SS$.

 The asymptotic estimates in Theorems \ref{Existence of dynamic rays} and \ref{Existence of parameter rays} show that dynamic and parameter rays have a well defined vertical order at infinity and that this order coincides with the order of their addresses in $\SS$. This is the analog of the cyclic order at infinity for polynomials rays.

For all these reasons we will refer to $\SS$ as the \emph{combinatorial space} for the family $e^z+c$.

For the family of unicritical polynomials $P_D$, the dynamic/parameter rays are in correspondence with the sequences over $D$ symbols (angles in D-adic expansion), that we can represent as 
\begin{small}
\begin{align*}
\SS_D&=\left\{ \frac{-D+1}{2},\dots,0,\dots,\frac{D-1}{2} \right\}^{ \N} &\text{for $D$ odd}\\
\SS_D&=\left\{\frac{-D+2}{2},\dots,0,\dots,\frac{D}{2} \right\}^{\N} &\text{for $D$ even}
\end{align*}
\end{small}

As for the exponential family, the dynamics of a unicritical polynomial of degree $D$ on the set of dynamic rays is conjugate to the the dynamics of the shift map $\sigma$ on $\SS_D$; also, dynamic and parameter rays have a cyclic order at infinity which corresponds to the cyclic order on $\SS_D$ if we identify the sequences modulo $D$.

 If $l,s\in \Z^\N$ are two sequences, $l=l_1 l_2\dots$ and $s=s_1 s_2 \dots$, we define the distance 
\[\text{dist}(l, s)=\sum_{s_k \neq l_k} \frac{1}{2^k},\]
 which turns $\SS$, $\SS_D$ into metric spaces.
  
The space $\SS_D$ embeds naturally in $\SS$ via the identity map; similarly, if $A\subset\SS$ is such that $|s|<N$ for each $s\in\SS$, then $A$ embeds in $\SS_D$ via the identity map for each $D>2N+2$.   

We will refer to $\SS_D$ as the \emph{combinatorial space} for the family $P_D$.

We will refer to  this description as \emph{combinatorial correspondence} between the exponential family and unicritical polynomials of sufficiently high degree $D$.


\section{Misiurewicz parameters}\label{Misiurewicz parameters}
Given the exponential family $f_c=e^z+c$, or a family of degree $D$ unicritical polynomials  $f_c=z^D+c$,
we call a  parameter $c_0$ \emph{Misiurewicz}  (or  \emph{postsingularly finite}) if the orbit of the singular/critical value is preperiodic.

In the exponential family,  the singular value is an omitted value and hence the postsingular orbit cannot be periodic; note also that such an orbit has to be repelling, otherwise the unique singular/critical  value would belong to the immediate attracting basin by a classical theorem of Fatou contradicting the fact that it is preperiodic.

 So being postsingularly finite is equivalent to say that the singular value $c_0$ lands a repelling orbit $\{z_i\}$ of period $m$ after $k$ iterations, for some integers $k,m$. We will refer to $\{z_i\}$ as the \emph{postsingular periodic orbit}.
 
From the definition  above and the discreteness of solutions of the equation $f_c^{k+m}(c)=f_c^k (c)$ it follows immediately that Misiurewicz parameters belong to the bifurcation locus. 

We will say that an exponential or polynomial map $f_{c_0}$ is \emph{Misiurewicz} (or \emph{postsingularly finite}) if $c_0$ is a Misiurewicz parameter.

There cannot be hyperbolic or parabolic Fatou components because $c$ belongs to the Julia set, nor Siegel disks because the orbit of $c$ accumulates on a finite set, so  for an exponential Misiurewicz map the Julia set is equal to $\C$.

There are  a few reasons why proving triviality of fibers (see section \ref{Fibers and Rigidity} for definition of fibers and a discussion on rigidity) for Misiurewicz parameters is easier than the other cases. Among them there is a correspondence between dynamical and parameter plane at Misiurewicz parameters. For polynomials, this is a well known result (see e.g. \cite{DH}, Chapter III, Theorem 2, for the quadratic case).

\begin{thm}\label{Dyn par correspondence polys}
Let $P_c^D$ be a unicritical polynomial with $c$ Misiurewicz. Then there are finitely many rays landing at $c$ in $\Pi^D_c$, whose angles are preperiodic. The parameter rays with the same angles land at $c$ in the parameter plane.
\end{thm},   

The analog of  Theorem \ref{Dyn par correspondence polys} for the exponential family is proven in \cite{SZ2}:

\begin{thm}{\bf Correspondence between dynamical and parameter plane.}\label{Correspondence between dynamical and parameter plane at Misiurewicz points}
An exponential  Misiurewicz parameter $c_0$ is the landing point of finitely many parameter rays  $G_{s_1},..,G_{s_q}$ whose addresses $s_1 <...<s_q$ are preperiodic of period $mq$ and preperiod $k$; moreover, the  dynamic rays $g_{s_1},...,g_{s_q}$ with the corresponding addresses land at $c_0$ in $\Pi_{c_0}$.
\end{thm}


Together with the fact that parameter rays with preperiodic address land (\cite{S3}), Theorem \ref{Correspondence between dynamical and parameter plane at Misiurewicz points} gives  a combinatorial classification of postsingularly finite exponential maps; see (\cite{LSV}, Theorem 2.6).
      
\begin{thm} {\bf Classification of Misiurewicz exponential maps.}\label{Classification of Misiurewicz exponential maps}
For every preperiodic  address $s$, there is a unique postsingularly finite exponential map such that the dynamic ray at address $s$ lands at the singular value. Every postsingularly finite exponential map is associated in this way to a positive finite number of preperiodic addresses.
\end{thm}

For unicritical polynomials it is well known that  parameter rays with preperiodic angles land at Misiurewicz parameters (\cite{PR}), so Theorem \ref{Classification of Misiurewicz exponential maps} offers a natural correspondence between exponential Misiurewicz parameters and polynomial Misiurewicz parameters through the angles/addresses of the parameter rays landing at them.
\vspace{1.5cm}
\begin{figure}[h]
\begin{center}$
\begin{array}{cc}
\includegraphics[width=6 cm]{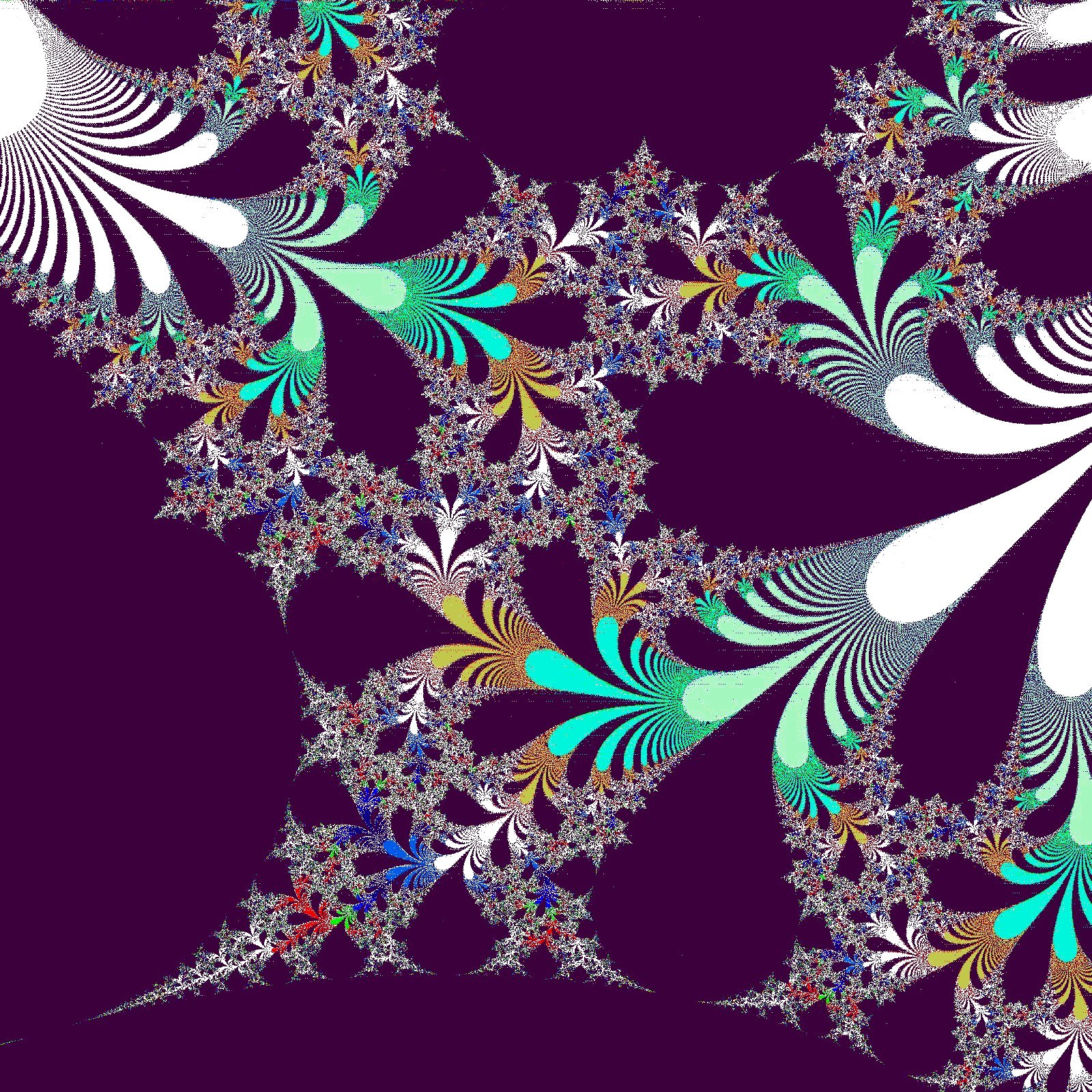} &
\includegraphics[width=6 cm]{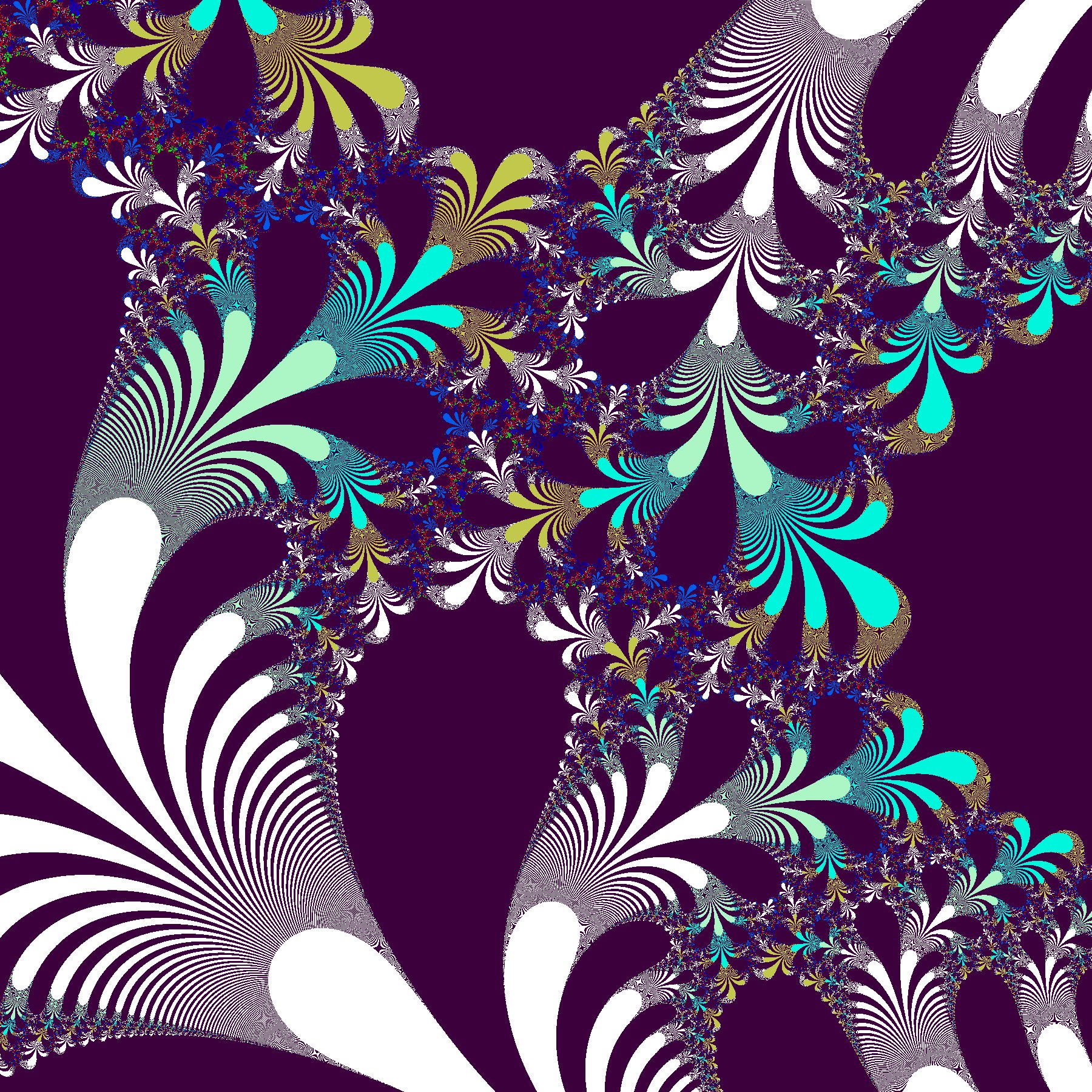}
\end{array}$
\end{center}
\caption{Parameter (on the left) and dynamical plane near the Misiurewicz parameter 1.81507+4.70945i; the spiralling of the  two rays landing at it can be inferred from the picture.}                  
\end{figure}
\vspace{0.5cm}

Before exploring further the consequences of the combinatorial classification of  Misiurewicz exponential maps, let us mention that the second main ingredient in proving triviality of fibers is offered by  contraction under the inverse map in a neighborhood of the postsingular periodic orbit.
\subsubsection*{A combinatorial property of Misiurewicz parameters}
One of the features of Misiurewicz parameters that we are going to use in the proof of our main theorem is a lemma connecting topology to combinatorics. It is  proved in \cite{SZ2} for exponentials and is probably known for unicritical  polynomials of degree $D$; for completeness we will include  a proof following the outline of \cite{SZ2}.
Before stating and proving the lemma, we need to introduce a dynamical partition for Misiurewicz parameters.

\vspace{4pt}
{\bf Dynamical partition.}
Let $f(z)=e^z+c_0$ or $f(z)=z^D+c_0$ where $c_0$ is a Misiurewicz parameter for the map under consideration, and
let $g_{s_1}$ be one of the finitely many dynamic rays landing at $c_0$ given by Theorem \ref{Correspondence between dynamical and parameter plane at Misiurewicz points} for exponentials and by Theorem \ref{Dyn par correspondence polys} for polynomials. 
The preimage of $g_{s_1}$ under $f$ is a set of countably many curves going to $-\infty$ in the case of exponentials, and a set of $D$ curves connecting at $0$ for a polynomial of degree $D$. In both cases, the preimages of $g_{s_1}$ partition the plane into open domains $W_j$.

Similarly, the preimages of $s_1$ under the shift map partition the combinatorial spaces $\SS$ and $\SS_D$ into the same number of sectors .
Label with the entry $0$ the dynamical and the combinatorial sector containing $c_0$ and $s_1$ respectively, and label all other sectors using consecutive integers respecting the cyclic/vertical order at infinity.

 Any non-escaping point, as well as any ray $g_s$ which does not belong to the backward orbit  of the ray  $g_{s_1}$, has a well defined itinerary whose entries keep track of the sectors visited by iterates of $s$ under the shift map.
 
We call this partition of the plane into the open domains $W_i$ a \emph{dynamical partition} for $f$. 

Different  choices of the dynamic ray landing at $c_0$ will lead to  different dynamical partitions; however, this choice will not matter to us.

For the sequel we will need to consider $\C\setminus\underset{n\geq-1}\cup f^n g_{s_1}$. This gives a new partition of the plane into domains $\widehat{W}_{i,j}$ where, for each fixed $i$,  $\widehat{W}_{i,j}$ denotes a connected component of $W_i$.

For convenience of the reader, let us recall that any domain $S$ of $\C$ whose complement contains at least two points, has for universal covering the unit disk. The standard hyperbolic metric of the unit disk can be pushed forward via a universal covering map to obtain a well defined hyperbolic metric on $S$ whose density we denote by $\rho_S$ (For details see e.g. Chapter 2 in \cite{M}).

We will use the following basic theorem about hyperbolic contraction (\cite{M}, Theorem 2.11):
\begin{thm}{\bf Schwarz-Pick lemma.}
If $f: S\ra S'$ is a holomorphic map between two domains admitting a hyperbolic metric, either $f$ is a local isometry, or $f$ strictly decreases all nonzero distances in the hyperbolic metrics of $S,S'$ respectively.
\end{thm}

A straightforward corollary is the so called \emph{monotonicity of the hyperbolic metric}:
\begin{lem}If $S,S'$ are domains admitting a hyperbolic metric with density $\rho_S,\rho_{S'}$ respectively and $S\subset S'$, then $\rho_{S'}(z)<\rho_S(z)$ for all $z\in S$.
\end{lem}

We are now ready to prove the required lemma. For exponentials this lemma appears already in \cite{SZ3}. We include the proof here, slightly modified, to also cover the case of unicritical polynomials.

\begin{lem}\label{Significance of dynamical partition for Misiurewicz parameters}{\bf{Significance of dynamical partition for Misiurewicz parameters.}}
Let $f$ be a Misiurewicz map, either exponential or a unicritical polynomial.
Then two (pre)periodic dynamic rays which are not preimages of the dynamic rays landing at the singular value land together if and only if they have the same itinerary with respect to the dynamical partition described above.
\end{lem}

\begin{proof}
Endow each of the domains $W_i$ and of the domains $\Wij$ with the corresponding hyperbolic metric.

For any domain $W_i$, let $f^{-1}_i$ be the inverse branch of $f$ mapping $\C\setminus g_{s_1}$ into $W_i$.
By the Schwarz Lemma, for any $i,i'$ the map $f^{-1}_{i'}|_{W_{i}}$ contracts the hyperbolic metric from $W_{i}$ to $W_{i'}$. By monotonicity of the hyperbolic metric, the  hyperbolic metric of each $\Wij$ is bigger than the hyperbolic metric of $W_i$, so $f^{-1}_i|_{\Wij}$ contracts the hyperbolic metric from $\Wij$ to $W_{i'}$.
 
Let us start by considering any two periodic dynamic rays which have the same itinerary with respect to the dynamic partition described in the previous section, and let $w_1$, $w_2$ be their periodic landing points. Up to selecting inverse branches of $f$, $w_1$, $w_2$ are  both  fixed under some $M$-th iterate $\Psi:=f_{a_1}^{-1}\circ\cdots\circ f_{a_M}^{-1}$ of the inverse of $f$, where the sequence $\ov{a_1..a_M}$ is equal to the common itinerary of $w_1$ and $w_2$.

If for at least one $j$, $1\leq j \leq M$, the $j$-th inverse iterate of $w_1$ and $w_2$ belong to the same $W_{i,a_j}$, and $\dist(w_1,w_2)\neq0$, we have that $\dist(\Psi(w_1),\Psi(w_2))<\dist(w_1,w_2)$,   which is a contradiction as  $\Psi(w_1)=w_1$ and $\Psi(w_2)=w_2$. 
 
This proves the theorem for periodic rays unless iterates of $w_1,w_2$ always belong to the same $W_i$ but  to different $\Wij$s.
 
So suppose that $w_1$ and $w_2$ belong to the same $W_i$ but to different $\Wij$. We will show that in this case $w_1, w_2$ have the same itinerary as a point of the postsingular periodic orbit and that this cannot happen.

Let us show that $w_1, w_2$ have the same itinerary as one of the postsingular periodic points. 
As $w_1, w_2$ belong to different $W_{a_1, i}$, for some postsingular periodic point $z$ one of them, say $w_1$, has to belong to one of the internal sectors defined in Section \ref{Combinatorics and Ray Portraits} (as they cannot be both in the external sector or they would belong to the same $W_{a_1, i}$). This implies that the first entry in the itinerary of $w_1$ (and hence in the itinerary of $w_2$, because the two itineraries are equal) coincides with the first entry in the itinerary of $z$. 
Using the fact that $f^j(w_1), f^j(w_2)$ belong to different $W_{a_j, i}$ for any $j$, the same  reasoning can be repeated to show that  $f^j(w_1), f^j(w_2)$ belong to the same $S_{a_j}$ as $f^j(z)$, hence that $w_1$ and $w_2$ have the same itinerary as $z$.
 
Remains to prove that no periodic point $w$ can have the same itinerary as a postsingular periodic point $z$. Suppose by contradiction  that this is the case and let $W_{i,j}$ be the domain such that $z,w\in\ov{W_{i,j}}$. 
As $z$  and $w$ have the same itinerary, we can find an inverse branch $\Phi$ of $f^{-k}$ fixing both $z$ and $w$, for some $k$. As $W_{i,j}$ does not intersect the postsingular set, $\Phi$ is well defined in all of $W_{i,j}$. If $L$ is a linearizing neighborhood of $z$, as $\Phi$ fixes $z$,   $\Phi^n\ra\{z\}$ on $W_{i,j}\cap L$, hence  $\Phi^n\ra\{z\}$ in $W_{i,j}$ by the identity principle, contradicting the fact that $\Phi(w)=w$.
 

%

%
   
Now let us consider preperiodic rays. If two preperiodic rays have the same itinerary, their periodic images also have the same itinerary, hence land together by previous part; and since the preperiodic rays are not preimages of the rays landing at the singular value, and they have the same itinerary, we can take pullbacks using the same branch for both, so that they keep landing together. 

On the other side if two rays land together they form a connected set, which never intersects the original partition under iterates of $f$, so they always belong to the same domain of the partition.

\end{proof} 


\section{Combinatorics and Ray Portraits}\label{Combinatorics and Ray Portraits}

This section introduces orbit portraits (in analogy with \cite{Mi} for polynomials and  \cite{RS2} for exponentials) and presents some theorems about the correspondence between  parameter rays with (pre)periodic angles for polynomials and  parameter rays with (pre)periodic addresses for exponentials.  
\begin{defn}
We call a \emph{ray pair}  any  couple of dynamic/parameter rays landing together whose addresses are (pre)periodic. 
\end{defn}
\begin{defn}
Let $ \{ z_i \} _{i=1\ldots n}$ be a repelling or parabolic periodic orbit of period $n$ in $\Pi_c$ ($\Pi^D_c$ resp.), and \[\AA_i:=\{r\in\SS \text{ ($\Suno$ resp.)}, r \text{ is periodic and } g^c_r \text{ lands at } z_i \}.\] Then $\PP=\{ \AA_1,\ldots,\AA_n \}$ is said to be the \emph{combinatorial orbit portrait} for $\{z_i \}$.

\end{defn}
Similarly,
\begin{defn}
Let $ \{ z_i \} _{i=1\ldots n}$ be a repelling or parabolic periodic orbit of period $n$ in $\Pi_c$ ($\Pi^D_c$ resp.), and $A_i:=\{g^c_r, g^c_r \text{ lands at } z_i \}$. Then $P=\{ A_1,\ldots\,A_n \}$ is said to be the \emph{orbit portrait} for $\{z_i \}$.

\end{defn}
The next two lemmas are given by Lemma 3.2 in \cite{RS2} for exponentials, and by Lemma 2.3 in \cite{M} for polynomials. 

\begin{lem}{\bf{Properties of combinatorial orbit portraits. }}
Given a combinatorial orbit portrait $\PP$, every $\AA_i\in\PP$ consists of a finite number of periodic addresses (angles resp.), and the shift map sends $\AA_i$ bijectively onto $\AA_{i+1}$. All addresses share the same period $qn$.
\end{lem}

\begin{lem}{\bf{Properties of  orbit portraits.}}
Given an orbit portrait $P$, every $A_i\in P$ consists of a finite number of dynamic rays, and $f$ maps $A_i$ bijectively onto $A_{i+1}$. All dynamic rays in the portrait are periodic with the same period $qn$.
\end{lem}

\begin{rem} In a more abstract way, we will speak of a combinatorial orbit portrait without specifying a periodic orbit $\{z_i\}$; such a combinatorial object is not necessarily \emph{realized}  (i.e occurs for some parameter) as an actual orbit portrait for some polynomial or exponential map.
\end{rem}

 The following theorems will show the  relation between which combinatorial orbit portraits are realized for exponential maps and which ones are realized for unicritical polynomials. As there are necessary and sufficient conditions for a combinatorial orbit portrait to be realized for polynomials, this give unique and sufficient conditions for a combinatorial portrait to be realized for exponentials.
 
Let us first state a correspondence between Misiurewicz parameters for exponential maps and for unicritical polynomials:
\begin{thm}\label{polyexp Misiurewicz}{\bf Misiurewicz addresses for exponentials and polynomials.}
The parameter rays $G_{s_1},\dots ,G_{s_q}$ land together at some exponential Misiurewicz parameter in the exponential parameter plane if and only if  for each family of unicritical polynomials of sufficiently high degree D the parameter rays with the same addresses land together at some polynomial Misiurewicz parameter.
\end{thm}
\begin{proof}
 Let  $G_{s_1},\dots ,G_{s_q}$ be the parameter rays landing together at some Misiurewicz parameter $c_0$ in exponential parameter plane. Then the dynamic rays with the corresponding addresses $g_{s_1},\dots ,g_{s_q}$ all have the same itineraries with respect to the dynamical partition induced by $g_{s_1}$ in $\Pi_{c_0}$, because together with $c_0$ they form a connected set whose orbit cannot intersect the boundaries of the dynamical partition. The address $s_1$ is preperiodic, so it is a sequence over finitely many values, so for polynomials of sufficiently high degree $D$ it represents the D-adic expansion of the  angle of some parameter ray. As  parameter rays with (pre)periodic angles  are well known to land for unicritical polynomials, there is a  Misiurewicz parameter $c_1$ depending on $D$ which is the landing point of the corresponding parameter ray. 
 
  By Theorem \ref{Correspondence between dynamical and parameter plane at Misiurewicz points}, the dynamic ray $g_{s_1}$ lands at $c_1$ in the polynomial dynamical plane for $f_{c_1}$.
  
 All the polynomial dynamic rays $g_{s_2},\dots, g_{s_q}$ also have the same itinerary with respect to the partition induced by $g_{s_1}$ so by Lemma \ref{Significance of dynamical partition for Misiurewicz parameters} they all land together in the dynamical plane for $f_{c_1}$. Then by \ref{Correspondence between dynamical and parameter plane at Misiurewicz points} the corresponding parameter rays land together at $c_1$ in the polynomial parameter plane.
\end{proof}
We will now define characteristic rays. We will state the definitions for exponentials, the corresponding definitions for polynomials can be inferred immediately.
\begin{defn}
Given an orbit portrait, the \emph{characteristic rays} are the rays $g_{s_1}, g_{s_2}$ which, together with their common endpoint, separate the singular value from all other rays in the portrait; compare with Lemma 3.3 in \cite{RS2} for existence and uniqueness.

  The \emph{characteristic sector} of an orbit portrait is given by the set of points enclosed between $g_{s_1}$ and  $g_{s_2}$.
\end{defn}

\begin{defn}
A \emph{characteristic ray pair} is  a pair of  parameter rays $G_{s_1}, G_{s_2}$ with periodic addresses landing together in parameter plane.
We will say that the domain enclosed by $G_{s_1}$ and $G_{s_2}$ (together with their commond endpoint) and which contains the rays of addresses between $s_1$ and $s_2$ is the \emph{wake} defined by $G_{s_1}, G_{s_2}$. 

\end{defn}
\begin{prop}{\bf Correspondence of bifurcations.}\label{Correspondence of bifurcations}
A parameter $c$ belongs to the wake defined by two parameter rays   $G_{s_1}, G_{s_2}$ together with their common landing point if and only if the dynamic rays $g_{s_1}, g_{s_2}$ land together in the dynamical plane $\Pi_c$ and are the characteristic rays for some orbit portrait in $\Pi_c$. We will call $s_1,s_2$ a pair of  \emph{characteristic addresses}.
\end{prop}
As corollary, $G_{s_1}, G_{s_2}$ is  a characteristic ray pair, landing together at some parameter $c$,  if and only if the dynamic rays $g_{s_1}, g_{s_2}$ are the characteristic rays of some orbit portrait in $\Pi_c$.
Remark: The theorem above is not explicitly stated in this way. It follows from Proposition 5.4 in \cite{RS2} when $c$ is hyperbolic or parabolic, and can be extended by holomorphic motions to all other parameters similarly as in \cite{R1}.

\begin{thm}\label{Correspondence of characteristic rays}{\bf Correspondence of characteristic rays.}
A pair of addresses is characteristic for exponentials if and only if it is characteristic for some unicritical polynomial of some degree $D$. In other words, given a combinatorial orbit portrait $\PP$, there exists  an exponential map $f_{\tilde{c}}$ realizing $\PP$ if and only if there is a unicritical polynomial $P^D_c$ realizing $\PP$, for some sufficiently high degree $D$.  
\end{thm}

\begin{rem} By the definitions and Proposition \ref{Correspondence of bifurcations} this will  show that if the parameter rays with periodic address $G_{s_1}, G_{s_2}$ land together in the parameter space  $\Pi^D_P$ for some $D$ then  the parameter rays with the corresponding addresses land together in the parameter  space of exponential maps $\Pi_P$; on the other side, if the parameter rays  $G_{s_1}, G_{s_2}$ land together in $\Pi_P$, they land together in $\Pi^D_P$ for all sufficiently high $D$.
\end{rem}

\begin{proof}
Let $\PP=\{\AA_i\}$ be a combinatorial orbit portrait. The inverse of the shift  map brings each non-characteristic sector to a sector bounded by rays whose addresses have the same first entry, and the characteristic sector to a sector bounded by rays for whose addresses the first entry differs by one.

So, being a characteristic sector is encoded in the topological orbit portrait, and the claim will follow if we can show that every combinatorial portrait is realized in the exponential family if and only if it is realized for $P_D$ of sufficiently high degree $D$.

If $\PP$ is realized for some polynomial in $\{P_c^D\}$, it persists in the whole wake bounded by its characteristic addresses so in particular it is  realized for some polynomial Misiurewicz parameter $c$ as well.

 This Misiurewicz parameter is the landing point of a dynamic ray of angle $s$, inducing a dynamical partition as described in the section \ref{Misiurewicz parameters}. The rays whose angles belong to the same $\AA_i$ land together, so they have the same itinerary with respect to this partition by \ref{Significance of dynamical partition for Misiurewicz parameters}; in particular, their angles have the same itineraries under the shift map with respect to the partition induced by $s$. 
  
By the combinatorial correspondence between polynomials and exponentials the angle $s$ in $D$-adic expansion can be seen as an address $s$ which  identifies a Misiurewicz parameter $\tilde{c}$ in the exponential family by \ref{Classification of Misiurewicz exponential maps}.  
All the dynamic  rays whose  addresses belong to $\PP$ exist in the dynamical plane of $\tilde{c}$, and  by Lemma \ref{Significance of dynamical partition for Misiurewicz parameters} they land together as they have the same itinerary with respect to the partition induced by the ray landing at the Misiurewicz parameter. 

If $\PP$ is realized for an exponential parameter, it persists in a wake by Theorem \ref{Correspondence of bifurcations},  so it is realized for some Misiurewicz parameter and can be transfered to a polynomial Misiurewicz parameter whose degree is sufficiently high to ensure the existence of the dynamic rays whose addresses belong to $\PP$ and of the dynamic ray landing at the Misiurewicz parameter. 
\end{proof}
\section{Fibers and Rigidity}\label{Fibers and Rigidity}

One of the main problems for one-parameter families in one-dimensional complex dynamics is to show that hyperbolic maps are dense. Whenever  the set of structurally stable parameters $S$ is dense, like for example for exponential maps and unicritical polynomials, saying that hyperbolic maps are dense is equivalent to saying that $S$ consists only of hyperbolic components.

 If there was a non-hyperbolic component, all maps in a neighborhood of a parameter in this component would be conjugate, so that any two maps in the component would have exactly the same set of combinatorial orbit portraits.
  By the theory of parabolic bifurcation in \cite{RS2}, this means that two parameters in the same non-hyperbolic component could not be separated by a parameter ray pair, or otherwise one of the two would have an additional orbit portrait. 
  This leads to the following definitions:  

\begin{defn}
The \emph{ parameter fiber} of a parameter $c_0 $ is the set of parameters which cannot be separated from $c_0$ by some pair of  parameter rays with (pre)periodic addresses landing together at a parabolic or Misiurewicz parameter, or by two parameter rays with periodic addresses landing at the boundary of the same hyperbolic component.

By analogy, the \emph{dynamical fiber} of a point $c_0 $ is the set of points which cannot be separated from $c_0$ by some pair of (pre)periodic rays landing together at some (pre)periodic point.
\end{defn}

\begin{defn}
We will say that the fiber of a point $c_0$ in dynamical/parameter space is \emph{trivial}, if any point $c\neq c_0$ can be separated from $c_0$ via a pair of (pre)periodic dynamic/parameter rays landing together (a ray pair), except for the points belonging to the dynamic/parameter rays which might land at the point $c_0$ itself.
\end{defn}
Note that the definition of fiber in dynamical and in parameter space are analogous by replacing dynamic rays with parameter rays.

We will call any result about triviality of fibers a \emph{rigidity result}.
 This comes from the fact that any map whose singular value does not escape and with trivial fiber cannot be conjugate to any other map in a neighborhood because two maps with different orbit portraits can not be topologically conjugate.
  
The next result follows immediately from the previous definitions. (See again \cite{RS3} for a slightly different formulation of this discussion.)

\begin{thm}
If the fiber of every non-hyperbolic, non-escaping parameter  is trivial, then every component in the set of structurally stable parameters is hyperbolic.
\end{thm} 

There are two main points in considering fibers to study density of hyperbolicity: for the exponential case, parameter rays with periodic address for exponentials are closely related to parameter rays for unicritical polynomials (see Theorem \ref{Correspondence of characteristic rays}), so that it is possible to infer results about exponentials using known results about polynomials; the second one, and more general one, is that fibers are a way to "localize" the global conjecture, and select  specific classes of parameters which are easier to study.

Our combinatorial rigidity statement (Theorem \ref{Triviality of fibers}) deals with the easiest class of parameters, the Misiurewicz parameters described in Section \ref{Misiurewicz parameters}. We restate it here for convenience:

\begin{thm}\label{Triviality of Misiurewicz}
Fibers of Misiurewicz parameters in parameter space are trivial, i.e. given any postsingularly finite parameter $c_0$, for any other parameter $c$ which does not belong to one of the finitely many parameter rays landing at $c_0$ there is a pair of parameter rays with periodic addresses landing together at a parabolic parameter which separate $c$ from $c_0$.
\end{thm}

\section{Triviality of Misiurewicz fibers}
In this section we prove Theorem \ref{Triviality of Misiurewicz}; for the remainder of the section, let $c_0$ be a Misiurewicz parameter. 
The proof follows the  outline of the corresponding result for polynomials (Lemma 7.1 and Theorem 7.3 in  \cite{S1}), using Theorem \ref{polyexp Misiurewicz} to establish a bridge between the combinatorics for polynomials and the combinatorics for exponentials. 
The following is Lemma 7.1 in \cite{S1}.

\begin{lem}\label{Combinatorial approximation polys}
Let $c_0$ be a Misiurewicz parameter for a family of unicritical polynomials, and let $G_{s_1},...,G_{s_q}$ be the parameter rays landing at $c_0$. Then  for all $ \epsilon>0$ there exist parameter ray pairs $P_i$ of angles  $(\alpha_i,\ \alpha_i')$ such that
$s_i<\alpha_i<\alpha_i'<s_{i+1}$ for $i=1,...,q-1$ and dist$(\alpha_i,s_i )<\epsilon$ ,  $\dist(\alpha_i',s_{i+1})<\epsilon$; moreover there is a  parameter ray pair $P_0$ of angles ($\alpha_0, \alpha_0')$ such that $\alpha_0< s_1< s_q<\alpha_0'$ and $\dist(\alpha_0,s_1 )<\epsilon$ ,  $\dist(\alpha_0',s_q)<\epsilon$.
\end{lem} 

We will use Lemma \ref{Combinatorial approximation polys} to prove the following exponential version.

\begin{prop}\label{Combinatorial approximation of parameter rays}{\bf Combinatorial approximation of parameter rays.}
Let $c_0$ be a Misiurewicz parameter for the exponential family, and let $G_{s_1},...,G_{s_q}$ be the parameter rays landing at $c_0$. Then  for all $ \epsilon>0$ there exist parameter ray pairs $P_i$ of angles  $(\alpha_i,\ \alpha_i')$ such that
$s_i<\alpha_i<\alpha_i'<s_{i+1}$ for $i=1,...,q-1$ and dist$(\alpha_i,s_i )<\epsilon$ ,  $\dist(\alpha_i',s_{i+1})<\epsilon$; moreover there is a  parameter ray pair $P_0$ of angles ($\alpha_0, \alpha_0')$ such that $\alpha_0< s_1< s_q<\alpha_0'$ and $\dist(\alpha_0,s_1 )<\epsilon$,  $\dist(\alpha_0',s_q)<\epsilon$.
\end{prop}
  
At first sight it might seem that this proposition would solve the  problem of triviality of fibers, but the relation between the "combinatorial topology" and the topology on $\C$ are far from clear, so we still have to show that the ray pairs which approximate the Misiurewicz rays combinatorially actually converge to them with respect to the standard topology on  $\C$ in a neighborhood of $c_0$. We will derive this from the following propositions:   
  
\begin{prop}\label{Triviality in dynamical plane}{\bf Triviality in dynamical plane. }
Let $c_0$ be a Misiurewicz parameter for the exponential family. Then dynamical fibers of the postsingular periodic orbit $\{z_i\}$ are trivial.
\end{prop}

Given a parameter $\ov{c}$ (for either unicritical polynomials or exponentials), and a repelling periodic orbit $\{z_i(\ov{c})\}$ with period $M$ and combinatorial orbit portrait $\PP$, for any $c$ in a sufficiently small parameter  neighborhood $U$ of $\ov{c}$ there are $M$ analytic functions $z_i(c)$ such that $\{z_i(c)\}$ is a repelling periodic orbit with  period $M$ and orbit portrait $\PP$ in the dynamical plane for $c$. We will call the orbit $\{z_i(c)\}$ the \emph{analytic continuation} of $\{z_i(\ov{c})\}$ (see \cite{M}, Appendix B). 
 
\begin{prop}\label{Persistence of dynamical triviality}{\bf{Persistence of dynamical triviality.}}
Let $c_0$ be as above. The postsingular periodic orbit $\{z_i\}$ has a well defined analytic continuation $\{z_i(c)\}$ for $c$ in a neighborhood of $c_0$, such that the dynamical fibers of $\{z_i(c)\}$ in $\Pi_c$ are trivial.
\end{prop}

At this point we will be able to prove our final theorem (equivalent to Theorem \ref{Triviality of Misiurewicz})

\begin{thm}\label{Triviality of Misiurewicz fibers}{\bf{Triviality of Misiurewicz fibers.}}
Let $c_0$ be a Misiurewicz parameter for the exponential family. Then any parameter $c$ can be separated from $c_0$ by a parameter ray pair, except for those parameters lying on the rays $G_{s_i}$ landing at $c_0$.
\end{thm}

\begin{proof}[Proof of Proposition \ref{Combinatorial approximation of parameter rays}: Combinatorial approximation of parameter rays]$\phantom{.}$

The core of the proof relies on the correspondence between combinatorial spaces for polynomials and for exponential parameters described in the end of Section \ref{Dynamic and parameter rays}; when the angles labeling rays for polynomials of degree $D$ are written in $D$-adic expansion as sequences over $D$ symbols, they can be seen as a subset of the exponentially bounded sequences encoding the combinatorics for exponential maps. 

Consider  the dynamic rays of addresses $s_1, \dots, s_q $ landing at the Misiurewicz parameter $c_0$ in $\Pi_\czero$. As noted in Lemma \ref{Significance of dynamical partition for Misiurewicz parameters}, each $g_{s_i}$ defines a partition with respect to which dynamic rays which are never mapped to $g_{s_i}$ have the same itinerary if and only if they land together in the dynamical plane. 

Also, $c_0$ is the landing point of the parameter rays $G_{s_1},...,G_{s_q}$. As $s_1, \dots, s_q $ include only finitely many symbols because they are finitely many preperiodic addresses, we can fix  a sufficiently high degree $D$ such that the parameter rays of angles $s_1,...,s_q$ all  exist for unicritical polynomials of degree $D$. Choose one of the addresses $s_1,...,s_q$, say $s_1$. As $s_1$ is preperiodic, the polynomial parameter ray  $G^D_{s_1}$ lands at some polynomial Misiurewicz parameter $\ctzero$ in the family $P^D_c$.

For $\ctzero$ the dynamic ray of angle $s_1$ also lands at the singular value in $\Pi^D_\ctzero$. All other angles $s_2,\dots, s_q$ have the same itinerary as $s_1$ with respect to the partition induced by  $s_1$ because they land together in $\Pi_{c_0}$, so by Lemma \ref{Significance of dynamical partition for Misiurewicz parameters} the dynamic rays $g_{s_1},...,g_{s_q}$ all land together at $\ctzero$. No other dynamic ray can land together with them, otherwise its angle would be an admissible sequence for exponentials and would have the same itinerary, so the corresponding exponential  ray would land together with $g_{s_1},\dots, g_{s_q}$ in the exponential dynamical plane as well. 

In the dynamical plane $\Pi_{\ctzero}^D$ by Lemma \ref{Combinatorial approximation polys} we have characteristic dynamic ray pairs approximating  each sector arbitrarily close. For any such ray pair of addresses $(\alpha, \alpha)$, the two rays in the ray pair have the same itinerary  by Lemma \ref{Significance of dynamical partition for Misiurewicz parameters}(with respect to the partition given by preimages of $g^D_{s_1}$), so the two angles $(\alpha, \alpha)$ have the same itinerary with respect to the partition induced in the combinatorial space $\SS_D$ by preimages of $s_1$. Hence the same addresses $(\alpha, \alpha)$ have the same itinerary in $\SS$, so that by  Lemma \ref{Combinatorial approximation polys} again the exponential  rays of addresses $(\alpha, \alpha)$ land together in  $\Pi_\czero$ giving the wanted approximating  ray pairs in the exponential dynamical plane. 

Let us  transfer the approximating dynamic ray pairs to the parameter plane for exponential maps. By Proposition \ref{Correspondence of bifurcations} as the approximating rays are characteristic the parameter rays with the corresponding addresses land together in the exponential parameter plane giving the wanted approximation for the sector defined by $G_{s_1}$ and $G_{s_q}$.
To approximate the other parameter sectors as well, fix a sector, say the sector between $G_{s_1}$ and $G_{s_2}$, call it $\widehat{s_1 s_2}$.

Let $V \subset \Pi_P$ be a neighborhood of $c_0$ such that there is an analytic continuation $\tilde{z}(c)$ of $c_0$ which keeps all the rays landing at $c_0$, and pick a Misiurewicz parameter $c$ in $V \cap \widehat{s_1 s_2}$. In $\Pi_c$ we will have the same relative position between $\tilde{z}$ and $c$ as we have in parameter plane between $c_0$ and $c$, in the sense that $c$ in $\Pi_c$ belongs to the sector defined by the rays of addresses $s_1$ and $s_2$: this follows from the fact that rays respect the vertical order induced by their addresses both in dynamical and in parameter plane.
 
Lemma \ref{Combinatorial approximation polys} gives characteristic dynamic ray pairs approximating $g^c_{s_1}$ and $g^c_{s_2}$ for polynomials (now $g^c_{s_1}$ and $g^c_{s_2}$ are landing at the repelling point $\tilde{z}(c)$, not at the singular value $c$); the corresponding rays can be obtained in the exponential dynamical plane by the same technique described above, and they  can be transfered in parameter plane by Proposition \ref{Correspondence of bifurcations}.    

\end{proof}
Note that this proposition  proves that we can separate a Misiurewicz parameter from all other Misiurewicz parameters, and from any parameter which is described combinatorially, for example parabolic and escaping parameters and landing points of parameter rays.

\begin{rem}
By the correspondence of characteristic ray pairs between polynomials and exponentials as stated in Theorem \ref{Correspondence of characteristic rays}, we could have obtained the combinatorial approximation directly in the parameter plane, but we need it also in dynamical plane in order to prove that dynamical fibers of the postsingular orbit are trivial and to proceed with the topological part of the proof. 
\end{rem}

\begin{proof}[Proof of Proposition \ref{Triviality in dynamical plane}: Triviality in dynamical plane]
Let $z$ be the first periodic point in the postsingular orbit, and $L$ be a  linearizing neighborhood. Let $k$ be the preperiod of $c_0$ and $m$ be the period of the postsingular periodic orbit, as in the definition of Misiurewicz parameters. Taking the $k$th image of  the approximating dynamic ray pairs found in the proof of  Proposition \ref{Combinatorial approximation of parameter rays} we obtain  dynamic ray pairs which approximate combinatorially the $q$ rays $g_{\sigma^k(s_1)}, \dots, g_{\sigma^k(s_q)}$ landing at $z$. We want to show that this combinatorial separation corresponds to an actual separation of all points in $L$ from $z$.

So for each sector defined by the $g_{\sigma^k(s_i)}$ consider an approximating ray pair which enters $L$. Note that such a ray pair must exist: as the Julia set $J(f_c)=\overline{I(f_c)}$ and $J(f_c)=\C$, any open set contains escaping points, so at least one ray must enter each sector, and once there is a ray inside it can be surrounded by one of the combinatorially approximating ray pairs.

Let $\psi$ be the branch of $f^{-m}$ fixing $z$,  and for a given sector call $V$ the closed region enclosed by the boundary of the sector, the boundary of $L$ and one of the approximating ray pairs entering that sector.  As $V$ does not contain any postsingular point except for $z$, which is fixed by $\psi$, $\psi^n$ is well defined for all $n$, and because we are in $L$ we have that  $\psi^n(V)\ra\{z\}$ as $n\ra\infty$.
 
\end{proof}

\begin{proof}[Proof of Proposition \ref{Persistence of dynamical triviality}: Persistence of dynamical triviality]
 Let $\{p_i\}_{i=1\dots q}$ 
 be the landing points of the ray pairs which enter the linearizing neighborhood in the proof of Propositon \ref{Triviality in dynamical plane}; then we can find a parameter neighborhood $V$ of $c_0$ in which we can continue analytically both the $p_i$'s and the postsingular periodic orbit $\{z_i\}$ with the same rays landing at them. 
 
Up to shrinking $V$, we can also assume that the rays enter the new linearizing neighborhood, and by contraction under the branch of $f^{-m}$ fixing $\{z_i(c)\}$, the neighborhoods between the approximating rays and the actual rays landing at $z_i(c)$ shrink to points like in the previous proposition.
 \end{proof}
 
 \begin{proof}[Proof of Theorem \ref{Triviality of Misiurewicz fibers}: Triviality of Misiurewicz fibers]
 Let us find  a parameter neighborhood V of $c_0$ so that every $c\in V$
 can be separated from $c_0$ by a parameter ray pair.
 
  Note that it is enough to separate from $c_0$ any parameter $c$ in the bifurcation locus, as rays cannot cross non-hyperbolic components.
  
 We use Propositions \ref{Triviality in dynamical plane} and \ref{Persistence of dynamical triviality} to show that the combinatorially approximating ray pairs found in Proposition \ref{Combinatorial approximation of parameter rays} converge to the rays landing at $c_0$ in the complex plane, so that the domains which we can separate combinatorially actually fill $V\setminus\cup G_{s_i}$. 
  
Like we did before in dynamical plane, let us distinguish the cases in which the parameter $c$ that we want to separate from $c_0$ is in the external sector which contains $-\infty$ (the one bounded by $G_{s_1}$ and $G_{s_q}$) and the case in which $c$ belongs to some of the other internal sector.

If $c$ belongs to the external sector, consider the dynamical plane for $c_0$, and separate $c$ from $c_0$ there by a preperiodic dynamic ray pair as from Proposition \ref{Triviality in dynamical plane} follows directly that the dynamical fiber of $c_0$ is trivial. Now separate this preperiodic ray pair by by one of the approximating characteristic ray pairs found in Proposition \ref{Combinatorial approximation of parameter rays} and then transfer this characteristic ray pair into parameter plane by Proposition \ref{Correspondence of bifurcations}. Note that $c$ in general does not have a ray landing at it. However the parameter ray pair and $c$ keep the same relative position in parameter plane that had $c$ and the corresponding dynamic rays in $\Pi_{c_0}$ by Proposition \ref{Correspondence of bifurcations}.

If $c$ belongs to one of the internal sectors,  say $\widehat{s_1 s_2}$, and also belongs to the neighborhood $V$ as in Proposition \ref{Persistence of dynamical triviality} then consider the dynamical plane $\Pi_c$. There, $c$ belongs to the corresponding dynamical sector $\widehat{s_1 s_2}$ defined at the analytic continuation $\tilde{z}(c)$. The dynamical fiber of $\tilde{z}$ is trivial by  Proposition \ref{Persistence of dynamical triviality}, so we can separate $c$ and $\tilde{z}(c)$ by some periodic ray pair ($\alpha,\alpha '$). This ray pair is persistent over a parameter neighborhood $U$ of $c$. This means that, for parameters in $U$, in dynamical plane  the  singular value will be inside the sector bounded by the  dynamical rays $(\alpha, \alpha')$. In particular,  by vertical order, escaping parameters in this neighborhood lie on a dynamic ray of address between $\alpha$ and $\alpha'$ in dynamical plane, so they lie on a parameter ray of address  between  $\alpha$ and $\alpha'$. By the combinatorial approximation given by Proposition \ref{Combinatorial approximation of parameter rays}, such a parameter is separated from $c_0$ by any of the ray pairs whose addresses are closer to $s_1$ and $s_2$ than $\alpha$ and $\alpha'$. This means that we can separate all those escaping parameters simultaneously from $c_0$ using the same ray pair $(\beta,\beta')$.  By density of escaping points in the bifurcation locus, we can approximate $c$ by escaping parameters, so the ray pair ($\beta, \beta'$) also separates $c_0$ from $c$ unless $c$ lies on $\beta$ or $\beta'$ in which case it has a well defined address and can be separated from $c_0$ by any ray pair closer than $\beta$ or $\beta'$.
\end{proof}

\begin{small}



\end{small}

\end{document}